\documentclass[amsmath,amssymb,aps,pra,reprint,groupedaddress,showpacs]{revtex4-1}

\usepackage{multirow}
\usepackage{verbatim}
\usepackage{color,graphicx,amsmath,amsfonts,mathtools,amsthm}

\theoremstyle{definition}
\newtheorem{definition}{Definition}[section]

\newtheorem{theorem}{Theorem}
\newtheorem{lemma}[theorem]{Lemma}

\newcommand{\padic}{$p$-adic }

\newcommand{\pq}{ \mathbb{Q}_p}

\begin{document}

\title{$p$-adic Ducci Sequences: a short note}

\author{Piero Giacomelli}
\email[]{pgiacomelli@fidiapharma.it}
\affiliation{Fidia Farmaceutici S.p.A.}

\date{\today}

\begin{abstract}
In this short note we formalized the definition for the Ducci operator $D$ in the context of the $p$-adic field $\mathbb{Q}_p$ as a natural extension of the classical one. Moreover we will describe the behavior of the operator and will provide some simple results as a counterpart to the classical one. 
\end{abstract}


\maketitle

\section{Introduction}

Ducci sequences were first introduced in 1937\cite{Ciamberlini1937}. Their attractiveness is due to the easy definition and to the interesting properties.
In the last years they regain attention and different questions regarding this sequences and the behaviour of the Ducci operator associate with them rise attentions.
From the basic definitions some research have been done on finding some extension on general abelian groups , on reals and on cyclotomic fields \cite{breuer2010ducci}. Some study tried to extend the definition to an higher dimension case \cite{breuer2010ducci}. 
Some interesting connection between Ducci sequence and cellular automata\cite{mendivil2012dynamics}. Surprisingly, being that the Ducci operator that define the Ducci sequence is based on absolute value norm, there have been no study at all on the Ducci sequences defined by the Ducci operator using the \padic norm. In these paper we introduced the \padic Ducci operator and we define the \padic Ducci sequences on non-Archimedean valued field  $\pq$. Moreover we will describe the behavior of the \padic Ducci operator like in the case of the classical Ducci sequences. 
\section{Previous results}\label{sec:review}
Let $n \in \mathbb{N}$, we define the Ducci operator $D$ as the operator that maps $\mathbb{Z}^N$ into itself  as follows:
\begin{align*}
    D:\mathbb{Z}^n & \rightarrow \mathbb{Z}^n \\
    (a_1,a_2,\dots,a_n) & \mapsto  D(a_1,a_2,\dots,a_n) = \\
    & (|a_1-a_2|,|a_2-a_3|,\dots,|a_n-a_1|)
\end{align*}

Let $k \in \mathbb{N}$, we also define that $D^k$ is the Ducci operator applied $k$ times. 
This brings to the following definition for the interation of the Ducci operator
The Ducci sequences are defined as the recurrent sequences with seed $\alpha=(a_1,a_2,\dots,a_n)$ and the following terms calculated by applying the Ducci operator $k$ times.
\begin{equation*}
\alpha_{(k)} =
\left\{
	\begin{array}{ll}
		\alpha_{(0)} = \alpha  & \mbox{if } k = 0 \\
		\alpha_{(k)} = D^k(\alpha) & \mbox{if } k > 0
	\end{array}
\right.
\end{equation*}
The previous literature focused on describing the behaviour of the Ducci sequences as $k \mapsto \infty$.
It has been proved that the Ducci sequences are ultimely periodic, so that there exists a number 
$m \in \mathbb{N}$ such that $D^k(\alpha) = D^{k+m}(\alpha)$. The number $m$ is called the lenght of the cycle.

Moreover it has been proved that if $n$ is a power of $2$ then there exists a value $K$ such that $D^k(\alpha)=0$ for every $k \geq K$

In these note we are interested on find if the same results holds one we redefine the Ducci operator in the context of the non-archimedian setting using the $p$-adic valuation in the definition of the Ducci operator.
\section{$p$-adic Ducci operator and sequences}
Let $p$ be a prime we are now ready to first define the $p$-adic Ducci operator.
The definition of the $p$-adic ducci Operator slitlghy differ from the absolute value because the result of the $p$-adic evaluation are always a integer power of $p$.
Let us starting by defining the $p$-adic Ducci operator $D_p$.

\begin{definition}
Let $P = \{0\} \cup \{p^i: i \in \mathbb{Z} \}$.  We can define the $p$-adic Ducci operator the following map.

\begin{align*}
    D_p:\mathbb{Q_p}^n & \rightarrow P^n \subseteq \mathbb{Q_p}^n \\
    (a_1,a_2,\dots,a_n) & \mapsto  D_p(a_1,a_2,\dots,a_n) = \\
    & (|a_1-a_2|_p,|a_2-a_3|_p,\dots,|a_n-a_1|_p).
\end{align*}
Where if $x \in \mathbb{Q}_p$ then $|x|_p = \frac{1}{p^{ord_p(x)}}$ being $ord_p(x) = max\{m: p^m | x\}$ (i.e. $ord_p(x)$ is the maximum power of $p$ that divide $x$).    

\end{definition}
From this definition if follows naturally the following one:

\begin{definition}
The $p$-adic Ducci sequences are the ones generated by $\alpha$ and applying the $p$-adic Ducci operator $k$-times with $k \in \mathbb{N}$, using the previous formalism if $\alpha = $
\begin{equation*}
\alpha^{(k)}_p =
\left\{
	\begin{array}{ll}
		\alpha^{(0)}_p = \alpha_p  & \mbox{if } k = 0 \\
		\alpha^{(k)}_p = D_p^k(\alpha) & \mbox{if } k > 0
	\end{array}
\right.
\end{equation*}
\end{definition}
We are interested in showing wich results holds in the context of the ultrametric inequality respect the context of the absolute value. In the classical settings the following simple hold.
\begin{itemize}
    \item{$D_p(0)=0$, where $0=\{0,0,\dots,0\}$}
    \item{$\forall a \in \mathbb{Q}_p$, $D_p(a\alpha)=a D_p(\alpha)$}
\end{itemize}

 the followings are true for the $p$-adic Ducci operator $D_p$ as well as for the Ducci operator $D$:
\begin{itemize}
    \item{$D_p(0)=0$, where $0=\{0,0,\dots,0\}$}
    \item{$\forall a \in \mathbb{Q_p}$, $D_p(a\alpha)=a D_p(\alpha)$}
\end{itemize}
The Ducci sequences are periodic being that for every sequence $\alpha^(k) = D^k(\alpha)$ there exists two natural indexes $r,s$ such that $\alpha^{(r)} = \alpha^{(r + s)}$. The number $r-s=c$ is called the lenght of the cycle. It is easy to see that in the classical context every constant sequence converge to the zeros sequence with cycle of lenght $c=1$. One difference respect to the classical setting is that if the starting seed of the $p$-adic Ducci sequence $\alpha^{(0)}$ is in the $p$-adic integer ring $\mathbb{Z}_p$ then being that $\alpha^{(k)} \in \{0,1\}^{n}, k > 1$ and that $|a_i-a_{i+1}|_p = max(|0|_p,| \pm{1}|_p) = max(|0|_p,|0|_p) = 0$ we have

\begin{equation*}
    \lim_{k \rightarrow +\infty} D_p^{(k)} = 0
\end{equation*}

so an easy lemma is the following one

\begin{lemma}
If $\alpha_p^{(0)} \in \mathbb{Z}_p$ the ducci sequences generated by this seed is the null sequence with period $1$.
\end{lemma}

\begin{lemma}
Let 
\begin{align*}
 \{\alpha^{(k)}_p\}_0^{\infty} & = \\
 & = \{\alpha^{(0)}_p,\alpha^{(1)}_p,\dots, \alpha^{(k)}_p,\dots \}   \\
 & = \{\alpha, D_p(\alpha), \\ D^2_p(\alpha) & = D_p(D_p(\alpha),\dots, \\  
 & D_p^k(\alpha),\dots \} 
\end{align*}

a $p$-adic Ducci sequence. The sequence is ultimely periodic. 

\end{lemma}

\begin{proof}

We first notice that apart from the first term $\alpha^{(0)}$ every term of the whole sequence $\{\alpha^{(k)}\}_0^{\infty}$ is in $P^n$. So let us consider a generic term $a_i^(k)$ of the $p$-adic ducci sequence, we can notice that by the ultrametric inequality in $x,y \in P$ then 
\begin{equation*}
|x-y|_p =
\left\{
	\begin{array}{ll}
		max(|x|_p,|y|_p)  & \mbox{if } x \neq y \\
		0  & \mbox{if } x = y
	\end{array}
\right.
\end{equation*}

so very term $a_i^{(k)}$ of the n-uple $\alpha^{(k)}$ is bounded between $p^{-\nu} \leq a_i^{(k)} \leq  p^{\nu}$ where $\nu = max(|a_i^{(0)}|_p)$. This means that there are only a finite number of possible values for $\alpha^{(k)}$ from $k>0$ in $P^n$. Then by the Pigeonhole principle \cite{ajtai1994complexity} there must be some $r,s \in \mathbb{N}, r,s > 0$ so that $\alpha^{(r)} = \alpha^{(r + s)}$. But then 

\begin{equation*}
    \alpha^{(r+s+ih)} = \alpha^{(r+i)}
\end{equation*}
for every $0\leq i < s$ and $h \in \mathbb{N}, h>0$.

\end{proof}

The first interesting result with the classical Ducci operator is the fact that if $\alpha = \alpha(0) = \{a_1,a_2,\dots,a_n\}$ contains a number of terms that is a power of $2$ then there exists a index $K \in \mathbb{N}$ such that $\alpha^{(k)} = 0$.

This follow from the observation that for every term $a^{(k)}$ then $|a_i -a_{i+1}|_p \equiv a_i+a_{i+1} \mod  2$. 

If $a_i \in P^n$ then up to rearrange the indexes 

\begin{align*}
|a_i-a_{i+1}|_p  & =|p^{\alpha}-p^{\beta}|_p  \\
& \leq max(|p^{\alpha}|_p,|p^{\beta}|_p) \\
& = min(\alpha,\beta) \\
\end{align*}

this means that in general $|a_i-a_{i+1}|_ \equiv |a_i+a_{i+1}| \mod{2}$ and so the study of the p-adic Ducci sequences reduce to the study of the period in $\mathbb{F}^n_2$. In particular the proof of section 3 in \cite{ehrlich1990periods} works without modification. 
So for example this means that in $n$ is a power of $2$ then for every p-adic Ducci sequence $\{\alpha^{(k)}_p\}$ there exist a $K \geq 0$ such that  $\{\alpha^{(k)}_p\}=0$ for every $k>K$.


\bibliography{p-adic-ducci-operator.bib}


\end{document}